\documentclass[12pt]{article}
\usepackage{amsmath,amsthm,amsfonts,amssymb,amscd,enumerate,graphicx}

\allowdisplaybreaks[4]
\newtheorem {Lemma}{Lemma}[section]
\newtheorem {Theorem} {Theorem}[section]
\newtheorem{Proposition}{Proposition}[section]
\newtheorem {Corollary}{Corollary}[section]

\numberwithin{equation}{section}

\begin{document}

\title{On the $\alpha$-spectral radius of graphs}

\author{Haiyan Guo\footnote{ghaiyan0705@163.com}, Bo Zhou\footnote{Corresponding author. E-mail: zhoubo@scnu.edu.cn}\\
School of  Mathematical Sciences, South China Normal University,\\
Guangzhou 510631, P.R. China}

\date{}
\maketitle

\begin{abstract}
For $0\le \alpha\le 1$, Nikiforov proposed to study the spectral properties of the family of matrices  $A_{\alpha}(G)=\alpha D(G)+(1-\alpha)A(G)$ of a graph $G$, where $D(G)$ is the degree diagonal matrix  and $A(G)$ is the adjacency matrix. The $\alpha$-spectral radius of $G$ is the largest eigenvalue of $A_{\alpha}(G)$.
We give upper bounds for $\alpha$-spectral radius for unicyclic graphs $G$ with maximum degree $\Delta\ge 2$, connected irregular graphs  with given maximum degree and  and some other graph parameters, and graphs with given domination number, respectively. We determine the unique tree with second maximum $\alpha$-spectral radius among trees, and the unique tree with maximum $\alpha$-spectral radius among trees with given diameter. For  a graph with two pendant paths at a vertex or at two adjacent vertex, we prove  results concerning the behavior of  the $\alpha$-spectral radius under relocation of a pendant edge in a pendant path. We also determine the unique graphs such that the difference between the maximum degree and the $\alpha$-spectral radius is maximum among trees, unicyclic graphs and non-bipartite graphs, respectively. 
\end{abstract}

\section{Introduction}

We consider simple and undirected graphs. Let $G$ be a
graph with vertex set $V(G)$ and edge set $E(G)$. Denote by $d_G(u)$ or simply $d_u$ the degree of $u$ in $G$. the degree of vertex $u$ in $G$.
Let $A(G)$ be the adjacency matrix and $D(G)$ the diagonal matrix of the degrees of $G$. The signless Laplacian matrix of $G$ is known as $Q(G)=D(G)+A(G)$.
The spectral properties of the adjacency matrix and the signless Laplacian matrix of a graph have been investigated for a long time, see, e.g., \cite{CDS, CRS}.
For any real $\alpha\in [0,1]$, Nikiforov~\cite{Ni1} proposed to study the spectral properties of the family of matrices  $A_{\alpha}(G)$ defined as the convex linear combination:
\[
A_{\alpha}(G)=\alpha D(G)+ (1-\alpha)A(G).
\]
Obviously, $A(G)=A_0(G)$ and $Q(G)=2A_{1/2}(G)$. For any real $\alpha\in [0,1]$, $A_{\alpha}(G)$ is a symmetric nonnegative matrix,
and thus its eigenvalues are all real. We call the largest eigenvalue of $A_{\alpha}(G)$ the $\alpha$-spectral radius of $G$,
denoted by $\rho_{\alpha}(G)$. Among others, Nikiforov \cite{Ni1} showed that the $r$-partite Tur\'an graph is the unique graph
with maximum $\alpha$-spectral radius for $0\le \alpha<1-\frac{1}{r}$ among $K_{r+1}$-free graphs on $n$ vertices
with $r\ge 2$, where $K_s$ is a complete graph with $s$ vertices. For   a tree $T$ with maximum degree $\Delta\ge 2$, Nikiforov et al.~\cite{NP} found an interesting bound for its $\alpha$-spectral radius:
$\rho_{\alpha}(T)<\alpha \Delta+2(1-\alpha)\sqrt{\Delta-1}$
when  $0\le \alpha<1$. This  implies some previous results in \cite{Go,St}.
They also showed in~\cite{NP} that for $0\le \alpha\le 1$, if $T$ is a tree on $n$ vertices, then $\rho_\alpha(P_n)\le \rho_\alpha (T)\le \rho_\alpha(S_n)$ with left (right, respectively) equality if and only if $T\cong P_n$ ($T\cong S_n$, respectively), where $S_n$ and $P_n$ are the star and the path on $n$ vertices, respectively.
Very recently, Nikiforov and Rojo \cite{NO} determined the unique graph with maximum $\alpha$-spectral radius among connected graphs on $n$ vertices with diameter (at least) $k$.


For $u,v\in V(G)$, the distance between $u$ and $v$ in $G$, denoted by $d_G(u,v)$, is the length of a shortest path
from $u$ to $v$ in $G$. The diameter of $G$ is the maximum distance between all vertex pairs of $G$.

A dominating set of $G$ is a vertex subset $S$ of $G$ such that each vertex of $V(G)\setminus S$ is adjacent to at least one vertex of $S$. The domination number of $G$, denoted by $\gamma(G)$, is the minimum cardinality of dominating sets of $G$.

In this article, we show that the upper bound for $\alpha$-spectral radius of trees with maximum degree $\Delta\ge 2$  in \cite{NP} holds also for unicyclic graphs, and
we give upper bounds for $\alpha$-spectral radius of connected irregular graphs  with fixed maximum degree and some other graph parameters, and of graphs with fixed domination number, respectively.
We determine the unique tree with second maximum $\alpha$-spectral radius among trees, and the unique tree with maximum $\alpha$-spectral radius among trees with given diameter. For  a graph with two pendant paths at a vertex or at two adjacent vertices,
we prove two results concerning the behavior of  the $\alpha$-spectral radius under relocation of a pendant edge in a pendant path, which were conjectured in \cite{NO}. We also determine the unique graphs such that the difference between the maximum degree and the $\alpha$-spectral radius is maximum among trees, unicyclic graphs and non-bipartite graphs, respectively.

\section{Preliminaries}


For a graph $G$ with $u\in V(G)$, $N_G(u)$ denotes the set of vertices that are adjacent to $u$ in $G$. For undefined notations and terminology for graphs, the readers are referred to \cite{BM}.

Let $G$ be a graph with $V(G)=\{v_1,\dots,v_n\}$. A column
vector $x=(x_{v_1},\dots, x_{v_n})^\top\in \mathbb{R}^n$ can be
considered as a function defined on $V(G)$ which maps vertex $v_i$
to $x_{v_i}$, i.e., $x(v_i)=x_{v_i}$ for $i=1,\dots,n$. Then
\[
x^\top A_{\alpha}(G)x=\alpha\sum_{u\in  V(G)}d_G(u)x_u^2+2(1-\alpha)\sum_{uv\in E(G)}x_ux_v.
\]
Moreover,  $\lambda$ is an eigenvalue of $A_{\alpha}(G)$ if and only if $x\ne 0$ and we have the following eigenequation at $u$ for each $u\in V(G)$:
\[
\lambda x_{u}=\alpha d_ux_{u}+(1-\alpha)\sum_{v\in N_G(u)}x_{v}.\]


If $0\le \alpha<1$ and  $G$ is connected, then $A_{\alpha}(G)$ is irreducible, and by Perron-Frobenius theorem,
it has a unique unit positive $x$ eigenvector corresponding to $\rho_{\alpha}(G)$. We call such a vector $x$ the Perron vector of $A_{\alpha}(G)$, see \cite{Ni1}.

If $G$ is connected, and $H$ is a proper subgraph of $G$, then by \cite[Crollary 2.2, p.~38]{Mi}, $\rho_\alpha(H) < \rho_\alpha(G)$ for  $0\le \alpha<1$.


The following lemma is somewhat similar to~\cite[Proposition~15]{Ni1}.

\begin{Lemma}\label{moving edge} \cite{NO} Let $G$ be a connected graph with $u,v\in V(G)$. Suppose that $v_1,\dots,v_s\in (N_G(v)\setminus N_G(u))\setminus\{u\}$, where $1\le s\le d_G(v)$. Let $G'=G-\{vv_i: 1\le i\le s\}+\{uv_i: 1\le i\le s\}$. Let  $0\le \alpha<1$ and let $x$ be  the Perron vector of $A_\alpha(G)$.  If $x_u\ge x_v$,
then $\rho_\alpha(G)<\rho_\alpha (G')$.
\end{Lemma}


\begin{Corollary} \label{ss} Let $G$ be a connected graph and $e=uv$  a cut edge of  $G$. Suppose that $G-\{e\}$ consists of two nontrivial
 components $G_1$ and $G_2$ with $u\in V(G_1)$ and $v\in V(G_2)$. Let $G'$ be a graph obtained from $G$ by identifying $u$ of $G_1$ with $v$ of $G_2$,
and adding a pendant edge to this common vertex. Then $\rho_\alpha(G)<\rho_\alpha(G')$ for $0\le \alpha<1$.
\end{Corollary}

\begin{proof} Let $x$ be  the Perron vector of $A_\alpha(G)$. We may assume that $x_u\ge x_v$. Let $N_{G_2}(v)=\{v_1,\dots,v_s\}$, where $s=d_G(v)-1\ge 1$.
Let $G^*=G-\{vv_i: 1\le i\le s\}+\{uv_i: 1\le i\le s\}$.
Obviously, $G^*\cong G'$. By Lemma~\ref{moving edge}, $\rho_\alpha(G)<\rho_\alpha(G^*)=\rho_\alpha(G')$.
\end{proof}

The following lemma is an extended version of Theorem~6.4.2 in \cite[p.~145]{CRSb}.

\begin{Lemma}\label{switching} Let $G$ be a connected graph  with edges $u_1u_2$ and $v_1 v_2$, where $u_1, u_2, v_1$ and $v_2$ are four distinct vertices of $G$, and $u_1v_2, v_1u_2\notin E(G)$. Let  $x$ the Perron  vector of $A_{\alpha}(G)$, where $0\le \alpha<1$. Let $G'=G-\{u_1u_2,v_1v_2\}+\{u_1v_2, v_1u_2\}$.
If $x_{u_1}\ge x_{v_1}$, $x_{u_2}\le x_{v_2}$ and one inequality is strict,   then $\rho_{\alpha}(G)<\rho_{\alpha}(G')$.
\end{Lemma}

\begin{proof} Note that
\begin{eqnarray*}
\rho_{\alpha}(G')-\rho_{\alpha}(G)&\ge& x^{\top}A_{\alpha}(G')x-x^{\top}A_{\alpha}(G)x\\
&=&2(1-\alpha)\sum_{uv\in E(G')}x_ux_v-2(1-\alpha) \sum_{uv\in E(G)}x_ux_v\\
&=&2(1-\alpha) (x_{v_1}x_{u_2}+x_{u_1}x_{v_2}-x_{u_1}x_{u_2}-x_{v_1}x_{v_2})\\
&=&2(1-\alpha)(x_{u_1}-x_{v_1})(x_{v_2}-x_{u_2})\\
&\ge& 0.
\end{eqnarray*}
Thus $\rho_{\alpha}(G')\ge \rho_{\alpha}(G)$.
Suppose that $\rho_{\alpha}(G')=\rho_{\alpha}(G)$. Then $x$ is the Perron  vector of $A_{\alpha}(G')$. We may assume that $x_{u_2}< x_{v_2}$.
From the eigenequations of $G'$ and $G$ at $u_1$, we have
\begin{eqnarray*}
\rho_{\alpha}(G')x_{u_1}&=&\alpha d_{u_1}x_{u_1}+(1-\alpha)\sum_{wu_1\in E(G')}x_w\\
&=&\alpha d_{u_1}x_{u_1}+(1-\alpha)\left(\sum_{wu_1\in E(G)}x_w-x_{u_2}+x_{v_2}\right)\\
&>&\alpha d_{u_1}x_{u_1}+(1-\alpha)\sum_{wu_1\in E(G)}x_w\\
&=&\rho_{\alpha}(G)x_{u_1},
\end{eqnarray*}
which is impossible. It follows that $\rho_{\alpha}(G')>\rho_{\alpha}(G)$.
\end{proof}

The following lemmas follows easily because as a quadratic function in $t$,  $at^2+b(t-c)^2$ for $a, b>0$ achieves its minimum value $\frac{abc^2}{a+b}$ when $t=\frac{bc}{a+b}$.

\begin{Lemma} \label{lem} \cite{Sh} If $a,b>0$, then $at^2+b(t-c)^2\ge \frac{abc^2}{a+b}$ with equality if and only if $t=\frac{bc}{a+b}$.
\end{Lemma}

\section{$\alpha$-spectral radius}

Let $B=(b_{ij})$ be an $n\times n$ nonnegative matrix with row sums $r_1, \dots, r_n$, where $r_1\ge\cdots\ge r_n$. Let $M$ be the largest diagonal entry and $N$ the largest non-diagonal entry of $B$, where $N>0$. Let $\rho(B)$ be the spectral radius of $B$. It is proved in \cite{DZ} that for $1\le \ell\le n$,
\[
\rho(B)\le \frac{r_{\ell}+M-N+\sqrt{(r_{\ell}-M+N)^2+4N\sum_{i=1}^{\ell-1}(r_i-r_{\ell})}}{2},
\]
with equality when $B$ is irreducible if and only if either  $r_1=\cdots=r_n$ or for some $2\le t\le \ell$, $b_{ii}=M$ for $1\le i\le t-1$, $b_{ik}=N$ for $1\le i\le n$ and $1\le k\le t-1$ with $k\ne i$, and $r_t=\cdots=r_n$. For a graph $G$ and $0\le \alpha<1$, we have $\rho_{\alpha}(G)=\rho(A_{\alpha}(G))$ and  applying this result in \cite{DZ}  to $A_{\alpha}(G)$, we have the following result.

Let $G$ be a graph on  $n\ge 2$  vertices with degree sequence $d_1, \dots, d_n$, where $d_1\ge \cdots\ge d_n$. Then for $0\le \alpha<1$ and $1\le \ell\le n$,
\[
\rho_{\alpha}(G)\le \frac{d_{\ell}+\alpha d_1-(1-\alpha)+\sqrt{(d_{\ell}-\alpha d_1+1-\alpha)^2+4(1-\alpha)\sum_{i=1}^{\ell-1}(d_i-d_{\ell})}}{2}
\]
with equality when $G$ is connected if and only if either $G$ is regular or $G$ is a graph with $d_1=\cdots=d_{t-1}=n-1>d_t=\cdots=d_n$ for some $2\le t\le \ell$.

By calculation of the $A_{\alpha}$-spectra of certain Bethe trees,  Nikiforov et al.~\cite{NP} showed that, for $0\le \alpha\le 1$,  $\rho_{\alpha}(T)<\alpha \Delta+2(1-\alpha)\sqrt{\Delta-1}$ for a tree $T$ with maximum degree $\Delta\ge 2$.  We extend this result to trees and unicyclic graphs.

\begin{Theorem}\label{t1}  Let $G$ be  a tree or unicyclic graph with maximum degree $\Delta\ge 2$. For $0\le \alpha\le 1$, we have
\[
\rho_{\alpha}(G)\le \alpha \Delta+2(1-\alpha)\sqrt{\Delta-1}
\]
with equality for $0\le \alpha<1$ if and only if $G$ is an cycle.
\end{Theorem}

\begin{proof}  If $\alpha=1$, then $A_{\alpha}(G)=D(G)$, and thus $\rho_{\alpha}(G)=\Delta=\alpha \Delta+2(1-\alpha)\sqrt{\Delta-1}$.

Suppose that $0\le \alpha<1$.

If $G$ is a tree, then we may add an edge between two vertices of degree one to form a unicyclic graph $G'$ with maximum degree $\Delta$, and for $0
\le \alpha<1$, by  \cite[Crollary 2.2, p.~38]{Mi}, we have
$\rho_{\alpha}(G)<\rho_{\alpha}(G')$. Thus we may assume that $G$ is a unicyclic graph. Let $x$ be the Perron vector of $A_{\alpha}(G)$.
Let  $C$ be the unique cycle of $G$ and $k$ its length. We label the vertices of $G$ such that $V(G)=\{v_1, \dots, v_n\}$ and $V(C)=\{v_1, \dots, v_k\}$.
For $w\in V(G)$, let $d_G(w, C)$ denote the minimum distance between $w$ and vertices  of $C$. We  orient  the  edges of $C$ as arcs $(v_1, v_2), \dots, (v_{k-1}, v_k), (v_k, v_1)$
and an edge $uv$ outside $C$  as $(u,v)$ if $d_G(u,C)>d_G(v,C)$. Now, for any $i=1, \dots, n$, there is a unique arc from $v_i$ to some other vertex $v_i'$.
Consider the multiple set $\{x_{v_1'}^2,\dots, x_{v_n'}^2\}$.  For $i=1, \dots, n$, the number of times of $x_{v_i}^2$ appearing in this multiple set is equal to the number
of arcs to $v_i$  under the above orientation, which is $d_G(v_i)-1$. Thus
\[
\sum_{i=1}^nx_{v_i'}^2=\sum_{i=1}^n(d_G(v_i)-1)x_{v_i}^2.
\]
Therefore
\begin{eqnarray}
\rho_{\alpha}(G)&=&x^\top A_{\alpha}(G)x \notag\\
&=&\alpha\sum_{u\in  V(G)}d_G(u)x_u^2+2(1-\alpha)\sum_{uv\in E(G)}x_ux_v \notag\\
&\le & \alpha\sum_{u\in  V(G)}\Delta x_u^2+2(1-\alpha)\sum_{uv\in E(G)}x_ux_v \label{e1}\\
&=&\alpha\Delta+2(1-\alpha)\sum_{i=1}^nx_{v_i}x_{v_i'} \notag\\
&\le & \alpha\Delta+2(1-\alpha)\sqrt{\sum_{i=1}^nx_{v_i}^2\sum_{i=1}^nx_{v_i'}^2} \label{e2}\\
&=& \alpha\Delta+2(1-\alpha)\sqrt{\sum_{i=1}^nx_{v_i'}^2} \notag\\
&=& \alpha\Delta+2(1-\alpha)\sqrt{\sum_{i=1}^n(d_G(v_i)-1)x_{v_i}^2} \notag \\
&\le & \alpha\Delta+2(1-\alpha)\sqrt{\sum_{i=1}^n(\Delta-1)x_{v_i}^2} \label{e3}\\
&=& \alpha\Delta+2(1-\alpha)\sqrt{\Delta-1}. \notag
\end{eqnarray}
In above inequalities, (\ref{e1}) and (\ref{e3}) follow from the fact that $d_G(u)\le \Delta$ for any $u\in V(G)$ and (\ref{e2}) follows from  Cauchy-Schwarz inequality.

If  $\rho_{\alpha}(G)=\alpha \Delta+2(1-\alpha)\sqrt{\Delta-1}$, then (\ref{e3}) is an  equality, implying that
 $G$ is $\Delta$-regular, and thus $\Delta=2$ and  $G=C$. Conversely, if $G$ is a cycle, then $\Delta=2$ and $\rho_{\alpha}(G)=2=\alpha \Delta+2(1-\alpha)\sqrt{\Delta-1}$.
\end{proof}

Let $G$ be a unicyclic graph with maximum degree $\Delta\ge 2$. By setting $\alpha=0, \frac{1}{2}$ in previous theorem respectively,
we have $\rho_0(G)\le 2\sqrt{\Delta-1}$ and
$\rho_{1/2}(G)\le \frac{1}{2}(\Delta+2\sqrt{\Delta-1})$ with either equality if and only if $G$ is a cycle.
The bound for $\rho_0(G)$  has been known in \cite{Hu}, and actually, we use techniques there.
Let $\mu(G)$ be the largest eigenvalue of the Laplacian matrix of a graph $G$. Note that $\mu(G)\le 2\rho_{1/2}(G)$
with equality if and only if $G$ is bipartite \cite{AM}. Thus $\mu(G)\le \Delta+2\sqrt{\Delta-1}$ with equality if and only if
$G$ is an even cycle, see \cite{Hu}.


If $G$ is a graph with maximum degree $\Delta$ and $0\le \alpha\le 1$, then $\rho_{\alpha}(G)\le \Delta$ with equality if and only if $\alpha=1$ or $G$ has a component that is regular of degree $\Delta$, see \cite[Proposition~11]{NP}.

For a connected irregular graph $G$ with $n$ vertices, maximum degree $\Delta$ and diameter $D$,
Cioab\v{a} \cite{Ci} proved a conjecture in \cite{CGNi} stated as
\[
\rho_{0}(G)<\Delta-\frac{1}{Dn},
\]
and Ning et al. \cite{NLL} showed  that  \[
2\rho_{1/2}(G)<2\Delta-\frac{1}{\left(D-\frac{1}{4}\right)n}.
\]

We follow the techniques in \cite{Ci,NLL} to prove the following result.


\begin{Theorem} \label{Gen} Let $G$ be a connected irregular graph on $n$ vertices with maximum degree $\Delta$ and diameter $D$. For $0\le \alpha<1$, we have
\[
\rho_\alpha(G)<\Delta-\frac{2(1-\alpha)}{(2D-\alpha)n}.\]
\end{Theorem}

\begin{proof} Let $x$ be  the Perron  vector of $A_{\alpha}(G)$. Let $x_z=\max\{x_i: i\in V(G)\}$. Then $x_z >\frac{1}{\sqrt{n}}$.

If $d_z<\Delta$, then from the eigenequation at $z$, we have
\begin{eqnarray*}
\rho_{\alpha}(G)x_z&=&\alpha d_zx_z+(1-\alpha)\sum_{j\in N_G(z)}x_j\\
&\le & \alpha d_zx_z+(1-\alpha)\sum_{j\in N_G(z)}x_z\\
&\le&\alpha(\Delta-1)x_z+(1-\alpha)(\Delta-1)x_z\\
&=&(\Delta-1)x_z,
\end{eqnarray*}
and thus $\rho_{\alpha}(G) \le \Delta-1<\Delta-\frac{2(1-\alpha)}{(2D-\alpha)n}$.

Assume that $d_z=\Delta$. Let $V_1=\{v\in V(G): d_v<\Delta\}$. Obviously, $V_1\neq \emptyset$.

Suppose that there is a vertex $u\in V_1$ such that $d_G(u,z)\le D-1$. Let $P=v_{0}v_{1}\dots v_{p}$ be a shortest path from $u$ to $z$, where $v_0=u$ and $v_p=z$. Then
\begin{eqnarray*}
\Delta-\rho_{\alpha}(G)&=&\Delta\sum_{i\in V(G)}x_{i}^2-x^{\top}A_{\alpha}x\\
&=&\sum_{i\in V(G)}(\Delta-d_i)x_{i}^{2}+(1-\alpha)\sum_{ij\in E(G)}(x_{i}-x_{j})^{2}\\
&\ge&x_{u}^2+(1-\alpha)\sum_{j=0}^{p-1}\left(x_{v_{j}}-x_{v_{j+1}}\right)^{2}.
\end{eqnarray*}
By Cauchy-Schwarz inequality and Lemma~\ref{lem}, we have 
\begin{eqnarray*}
\Delta-\rho_{\alpha}(G)
&\ge&x_{u}^2+\frac{(1-\alpha)(x_u-x_z)^2}{p}\\
&\ge& \frac{1-\alpha}{p+1-\alpha}x_z^2\\
&\ge & \frac{1-\alpha}{D-\alpha}x_z^2\\
&>& \frac{2(1-\alpha)}{(2D-\alpha)n},
\end{eqnarray*}
as desired.

Now assume that for every vertex $v\in V_1$, $d(v,z)=D$. We consider the cases $|V_1|\ge 2$ and $|V_1|=1$ separately.

Suppose first that $|V_1|\ge 2$. Let $u,v\in V_1$, and $P=v_{0}v_{1}\dots v_{D}$ be a shortest path from $u$ to $z$, where $v_0=u$ and $v_D=z$.  Let $Q$ be a shortest path from $v$ to $z$. Let $\ell=\min\{j: v_j\in V(Q)\}$. Then $\ell\in \{1,\dots, D\}$ and $\ell=d_G(u,v_\ell)=d_G(v,v_\ell)$. Let $Q_{v,v_\ell}$ be the sub-path of $Q$ from $v$ to $v_\ell$.
%
If $\ell\ne D$, i.e., $\ell\le D-1$, then by Cauchy-Schwarz inequality and Lemma~\ref{lem}, we have
\begin{eqnarray*}
\Delta-\rho_{\alpha}(G)&=&\sum_{i\in V(G)}(\Delta-d_i)x_{i}^{2}+(1-\alpha)\sum_{ij\in E(G)}(x_{i}-x_{j})^{2}\\
&\ge&x_{u}^2+x_{v}^2+(1-\alpha)\left(\sum_{j=0}^{\ell-1}\left(x_{v_{j}}-x_{v_{j+1}}\right)^{2}\right.\\
&&\left.+\sum_{kj\in E(Q_{v,v_\ell})}(x_{k}-x_{j})^{2}+\sum_{j=\ell}^{D-1}\left(x_{v_{j}}-x_{v_{j+1}}\right)^{2}\right)\\
&\ge&\left(x_{u}^2+\frac{(1-\alpha)\left(x_{u}-x_{v_{\ell}}\right)^{2}}{\ell}\right)+\left(x_{v}^2+\frac{(1-\alpha)\left(x_{v}-x_{v_{\ell}}\right)^{2}}{\ell}\right)\\
&&+\frac{(1-\alpha)\left(x_{v_\ell}-x_z\right)^{2}}{D-\ell}\\
&\ge& \frac{2(1-\alpha)}{\ell+1-\alpha}x_{v_\ell}^2+\frac{(1-\alpha)\left(x_{v_\ell}-x_{z}\right)^{2}}{D-\ell}\\
&\ge& \frac{2(1-\alpha)}{2D-\ell+1-\alpha}x_{z}^2\\
&\ge & \frac{2(1-\alpha)}{2D-\alpha}x_{z}^2\\
&>& \frac{2(1-\alpha)}{(2D-\alpha)n}.
\end{eqnarray*}
%
%
If $\ell=D$, i.e., $v_\ell=z$, then as above and noting that $D>1$, we have
\begin{eqnarray*}
\Delta-\rho_{\alpha}(G)&=&\sum_{i\in V(G)}(\Delta-d_i)x_{i}^{2}+(1-\alpha)\sum_{ij\in E(G)}(x_{i}-x_{j})^{2}\\
&\ge&x_{u}^2+x_{v}^2+(1-\alpha)\left(\sum_{j=0}^{D-1}\left(x_{v_{j}}-x_{v_{j+1}}\right)^{2}+\sum_{ij\in E(Q)}(x_{i}-x_{j})^{2}\right)\\
&\ge&\left(x_{u}^2+\frac{(1-\alpha)(x_{u}-x_{z})^{2}}{D}\right)+\left(x_{v}^2+\frac{(1-\alpha)(x_{v}-x_{z})^{2}}{D}\right)\\
&\ge& \frac{2(1-\alpha)}{D+1-\alpha}x_{z}^2\\ 
&>& \frac{2(1-\alpha)}{(2D-\alpha)n}.
\end{eqnarray*}
Thus, the result follows when $|V_1|\ge 2$.

Now assume  that  $|V_1|=1$. Let $w$ be a vertex of $G$ such that $x_w=\min\{x_i: i\in V(G)\}$. Since
\[
\Delta x_w>\rho_{\alpha}(G)x_w=\alpha d_{w}x_{w}+(1-\alpha)\sum_{j\in N_G(w)}x_j\ge d_wx_w,\]
we have $d_w<\Delta$, implying that $V_1=\{w\}$.

Since $\rho_{\alpha}(G)x_i=\alpha d_{i}x_{i}+(1-\alpha)\sum_{j\in N_G(i)}x_j$ for $i\in V(G)$, we have
\begin{eqnarray*}
\rho_{\alpha}(G)\sum_{i\in V(G)}x_i&=&\alpha\sum_{i\in V(G)} d_{i}x_{i}+(1-\alpha)\sum_{i\in V(G)} \sum_{j\in N_G(i)}x_j\\
&=&\sum_{i\in V(G)}d_ix_i=\Delta \sum_{i\neq w}x_i+d_wx_w,
\end{eqnarray*}
i.e., $(\Delta-\rho_{\alpha}(G))\sum_{i\in V(G)}x_i=(\Delta-d_w)x_w$, from which we get
\[
\Delta-\rho_{\alpha}(G)=\frac{(\Delta-d_w)x_w}{\sum_{i\in V(G)}x_i}>\frac{x_w}{nx_z}.
\]

Let $\gamma=\frac{x_z}{x_w}$. If  $\gamma\le\frac{2D-\alpha}{2(1-\alpha)}$,
then
\[
\Delta-\rho_{\alpha}(G)> \frac{1}{n\gamma}\ge\frac{2(1-\alpha)}{(2D-\alpha)n},
\]
as desired.

In the following, we assume that $\gamma>\frac{2D-\alpha}{2(1-\alpha)}$.

Since $d_G(w,z)=D$, we can choose a vertex $z'\in N_G(z)$   such that $d_G(w,z')=D-1$.  Let
$v_0\dots v_{D-1}$ be a shortest path from $w$ to $z'$ with $v_0=w$ and $v_{D-1}=z'$. Then as above, we have
\begin{eqnarray*}
\Delta-\rho_{\alpha}(G)&=&x_{w}^2+(1-\alpha)\sum_{ij\in E(G)}(x_{i}-x_{j})^{2}\\
&\ge&x_{w}^2+(1-\alpha)\sum_{j=0}^{D-2}\left(x_{v_{j}}-x_{v_{j+1}}\right)^{2}\\
&\ge&x_{w}^2+\frac{(1-\alpha)(x_w-x_{z'})^2}{D-1}\\
&\ge& \frac{1-\alpha}{D-\alpha}x_{z'}^2.
\end{eqnarray*}

If  $x_{z'}>\frac{1}{\sqrt{n}}$, then
\[
\Delta-\rho_{\alpha}(G)> \frac{1-\alpha}{(D-\alpha)n}
\ge  \frac{2(1-\alpha)}{(2D-\alpha)n},
\]
as desired.

Thus, we assume that there is a vertex $z'\in N_G(z)$  such that $x_{z'}\le \frac{1}{\sqrt{n}}$. Then
\[
\rho_{\alpha}(G)x_z=\alpha \Delta x_{z}+(1-\alpha)\sum_{i\in N_G(z)}x_i\le (\Delta-1+\alpha) x_{z}+(1-\alpha)\frac{1}{\sqrt{n}},
\]
which implies $\Delta-\rho_{\alpha}(G)\ge (1-\alpha)\left(1-\frac{1}{x_z\sqrt{n}}\right)$.

If $(1-\alpha)\left(1-\frac{1}{x_z\sqrt{n}}\right)> \frac{2(1-\alpha)}{(2D-\alpha)n}$, then we are done. Thus, we assume that
$(1-\alpha)(1-\frac{1}{x_z\sqrt{n}})\le \frac{2(1-\alpha)}{(2D-\alpha)n}$, i.e.,
\[
x_z\le \frac{(2D-\alpha)\sqrt{n}}{(2D-\alpha)n-2}.
\]
This, together with the fact that $(n-1)x_z^2+x_w^2\ge\sum_{i\in V(G)}x_i^2=1$, implies that

\[
\gamma^2=\left(\frac{x_z}{x_w}\right)^2\le \frac{1}{\frac{1}{x_z^2}-(n-1)}\le \frac{(2D-\alpha)^2n}{(2D-\alpha)(2D-\alpha-4)n+4}.
\]

If $D\ge 3$, then $(2D-\alpha)(2D-\alpha-4)n+4\ge (\alpha^2-8\alpha+12)n+4>4n$, and thus
\[
\gamma^2<\frac{(2D-\alpha)^2n}{4n}\le \frac{(2D-\alpha)^2}{4(1-\alpha)^2}<\gamma^2,\]
which is a contradiction.
Thus, it follows that $D=2$. 

By Lemma~\ref{lem} and the fact that $x_z^2>\frac{1}{n}$, 
we have $x_{w}^2+(1-\alpha)(x_w-x_z)^2>\frac{1-\alpha}{2-\alpha}x_z^2>\frac{2(1-\alpha)}{(2D-\alpha)n}$.

Suppose that there are  two paths, say $wuz$ and $wvz$, from $w$ to $z$. Note that $(x_w-t)^{2}+(t-x_z)^{2}\ge \frac{1}{2}(x_w-x_z)^2$. As earlier,   we have
\begin{eqnarray*}
\Delta-\rho_{\alpha}(G)&=&x_{w}^2+(1-\alpha)\sum_{ij\in E(G)}(x_{i}-x_{j})^{2}\\
&\ge&x_{w}^2+(1-\alpha)\left((x_w-x_u)^{2}+(x_{u}-x_z)^{2}+(x_w-x_{v})^{2}+(x_{v}-x_{z})^{2}\right)\\
&\ge&x_{w}^2+(1-\alpha)(x_w-x_z)^2\\
&>&\frac{2(1-\alpha)}{(2D-\alpha)n},
\end{eqnarray*}
as desired.

Thus, we  assume that there is a unique path, say $wuz$,  from $w$ to $z$. Let $N_1=N_G(z)\setminus \{u\}$ and  let $N_2$ the set of vertices of distance $2$ from $z$ except $w$. Then $V(G)\setminus\{z,u,w\}=N_1\cup N_2$, and   for every vertex $v\in N_1$, $d_G(v,w)=2$. We consider three cases.

\noindent \textbf{Case 1.} $u$ is adjacent to at least two vertices in $N_1$.

We choose $v,v'\in N_1\cap N_G(u)$. Since $d_u=d_z=\Delta$, there is a vertex in $N_1\setminus\{v,v'\}$, say $v_1$, such that $uv_1\notin E(G)$. Note that $d_G(v_1, w)=2$. Then there is a path, say $v_1v_2w$, connecting $v_1$ and $w$, where $v_2\in N_2$. Then as earlier, we have
\begin{eqnarray*}
\Delta-\rho_{\alpha}(G)&=&x_{w}^2+(1-\alpha)\sum_{ij\in E(G)}(x_{i}-x_{j})^{2}\\
&\ge&x_{w}^2+(1-\alpha)\left((x_{w}-x_{u})^{2}+(x_{u}-x_{z})^{2}+(x_{u}-x_{v})^{2}+(x_{v}-x_{z})^{2} \right.\\
&&+\left.(x_{u}-x_{v'})^{2}+(x_{v'}-x_{z})^{2}\right.\\
&&+\left.(x_{w}-x_{v_2})^{2}+(x_{v_2}-x_{v_1})^{2}+(x_{v_1}-x_{z})^{2}\right)\\
&\ge&x_{w}^2+(1-\alpha)\left((x_{w}-x_{u})^{2}+(x_{u}-x_{z})^{2}+\frac{(x_u-x_z)^2}{2}\right.\\
&&\left.+\frac{(x_u-x_z)^2}{2}+\frac{(x_{w}-x_{z})^{2}}{3}\right)\\
&=&x_{w}^2+(1-\alpha)\left((x_{w}-x_{u})^{2}+2(x_{u}-x_{z})^{2}+\frac{(x_{w}-x_{z})^{2}}{3}\right).
\end{eqnarray*}
By Lemma~\ref{lem}, 
\[
\Delta-\rho_{\alpha}(G)\ge x_{w}^2+(1-\alpha)(x_w-x_z)^2> \frac{2(1-\alpha)}{(2D-\alpha)n},
\]
as desired.

\noindent \textbf{Case 2.} $u$ is adjacent to  exactly one vertex in $N_1$.

Let $v$ be the unique vertex in $N_1\cap N_G(u)$.
Since $d_u=d_z=\Delta$, there is a vertex in $N_1\setminus\{v\}$, say $v_1$, such that $uv_1\notin E(G)$. Note that $d_G(v_1, w)=2$. Then there is a path, say $v_1v_2w$, connecting $v_1$ and $w$£¬ where $v_2\in N_2$. Since $w, u, z, v, v_1, v_2\in V(G)$, we have $n\ge 6$. If $n=6$, then  $\Delta=d_z=3$, $d_w=2$, and thus $2|E(G)|=5\Delta+d_w=17$, a contradiction. Thus $n\ge 7$. 

\noindent \textbf{Case 2.1.} $\Delta\ge 4$.

Since $|N_1|=d_z-1=\Delta-1\ge 3$, we may choose $s\in N_1\setminus\{v,v_1\}$. Since $D=2$, there is a path, say $ss'w$, connecting $s$ and $w$, where $s'\in N_2$.

If $s'=v_2$, then as above, we have
\begin{eqnarray*}
\Delta-\rho_{\alpha}(G)&=&x_{w}^2+(1-\alpha)\sum_{ij\in E(G)}(x_{i}-x_{j})^{2}\\
&\ge&x_{w}^2+(1-\alpha)\left((x_{w}-x_{u})^{2}+(x_{u}-x_{z})^{2}\right.\\ 
&&+\left.(x_{w}-x_{v_2})^{2}+(x_{v_2}-x_{v_1})^{2}+(x_{v_1}-x_{z})^{2}\right.\\
&&+\left.(x_{v_2}-x_{s})^{2}+(x_{s}-x_{z})^{2}\right)\\
&\ge&x_{w}^2+(1-\alpha)\left(\frac{(x_w-x_z)^2}{2}+(x_{w}-x_{v_2})^{2}+(x_{v_2}-x_z)^2\right)\\
&\ge&x_{w}^2+(1-\alpha)\left(\frac{(x_w-x_z)^2}{2}+\frac{(x_w-x_z)^2}{2}\right)\\
&=&x_{w}^2+(1-\alpha)(x_w-x_z)^2\\
&>& \frac{2(1-\alpha)}{(2D-\alpha)n},
\end{eqnarray*}
as desired.

If $s'\neq v_2$, then as above, we have
\begin{eqnarray*}
\Delta-\rho_{\alpha}(G)&=&x_{w}^2+(1-\alpha)\sum_{ij\in E(G)}(x_{i}-x_{j})^{2}\\
&\ge&x_{w}^2+(1-\alpha)\left((x_{w}-x_{u})^{2}+(x_{u}-x_{z})^{2}\right.\\ 
&&+\left.(x_{w}-x_{v_2})^{2}+(x_{v_2}-x_{v_1})^{2}+(x_{v_1}-x_{z})^{2}\right.\\
&&+\left.(x_{w}-x_{s'})^{2}+(x_{s'}-x_{s})^{2}+(x_{s}-x_{z})^{2}\right)\\
&\ge&x_{w}^2+(1-\alpha)\left(\frac{(x_w-x_z)^2}{2}+\frac{2(x_w-x_z)^2}{3}\right)\\
&\ge&x_{w}^2+\frac{7(1-\alpha)(x_w-x_z)^2}{6}\\
&>&x_{w}^2+(1-\alpha)(x_w-x_z)^2\\
&>&\frac{2(1-\alpha)}{(2D-\alpha)n},
\end{eqnarray*}
as desired.

\noindent \textbf{Case 2.2.} $\Delta=3$.

Suppose that $n\ge 8$. Then there are two vertices, say $s_1, s_2\in V(G)\setminus\{w,u,z,v,v_1,v_2\}$. Since $D=2$, $d_u=3$ and $d_w=2$,   we have  $d_G(s_1, w)=d_G(s_2, w)=2$, and thus  $s_1$ and $s_2$ can only be adjacent to $v_2$, which is impossible because  $d_{v_2}=3$. Thus $n=7$. Let $s$ be the vertex different from $w,u,z,v,v_1,v_2$. Then
 $E(G)=\{wu,uz,uv,vz,wv_2,v_{2}v_{1},v_{1}z,$ $v_{2}s,sv,sv_1\}$, see  Fig.~\ref{fig:1}.
\begin{figure}[htbp]
\centering
\includegraphics[width=5cm,height=4cm]{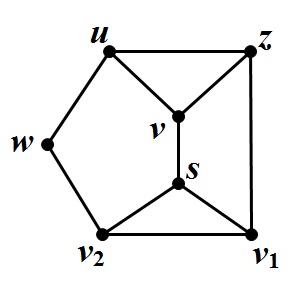}
\caption{The only possible graph $G$  in Case 2.2.}
\label{fig:1}
\end{figure}
Note that there is an automorphism $\sigma$ such that $\sigma(s)=z$.
By \cite[Proposition 16]{Ni1},  
$x_s=x_z$. Thus as above, we have
\begin{eqnarray*}
\Delta-\rho_{\alpha}(G)&=&x_{w}^2+(1-\alpha)\sum_{ij\in E(G)}(x_{i}-x_{j})^{2}\\
&\ge&x_{w}^2+(1-\alpha)\left((x_{w}-x_{u})^{2}+(x_{u}-x_{z})^{2}+(x_{w}-x_{v_2})^{2}+(x_{v_2}-x_{s})^{2}\right)\\
&\ge &x_{w}^2+(1-\alpha)(x_w-x_z)^2\\
&>&\frac{2(1-\alpha)}{(2D-\alpha)n},
\end{eqnarray*}
as desired.

\noindent \textbf{Case 3.} $u$ is not adjacent to any vertex in $N_1$.

Since $d_w<\Delta$, there are two vertices, say  $v_1$ and $v_2$ in $N_1$, such that some vertex $v^*$ in $N_2$ is adjacent to $w$, $v_1$ and $v_2$.
Thus, we have
\begin{eqnarray*}
\Delta-\rho_{\alpha}(G)&=&x_{w}^2+(1-\alpha)\sum_{ij\in E(G)}(x_{i}-x_{j})^{2}\\
&\ge&x_{w}^2+(1-\alpha)\left((x_{w}-x_{u})^{2}+(x_{u}-x_{z})^{2}+(x_w-x_{v^*})^{2}+(x_{v^*}-x_{v_1})^{2}\right.\\
&&+\left.(x_{v_1}-x_{z})^{2}+(x_{v^*}-x_{v_2})^{2}+(x_{v_2}-x_{z})^{2}\right)\\
&\ge & x_{w}^2+(1-\alpha)\left(\frac{(x_w-x_z)^2}{2}+(x_w-x_{v^*})^{2}+\frac{(x_{v^*}-x_z)^2}{2}+\frac{(x_{v^*}-x_z)^2}{2}\right)\\
&\ge&x_{w}^2+(1-\alpha)\left(\frac{(x_w-x_z)^2}{2}+\frac{(x_w-x_z)^2}{2}\right)\\
&=& x_w^2+(1-\alpha)(w_w-x_z)^2\\
&>&\frac{2(1-\alpha)}{(2D-\alpha)n},
\end{eqnarray*}
as desired.

Now by combining the above three cases, we complete the proof.
\end{proof}

Besides those considerations in \cite{Ci, NLL}, the  proof of  Theorem~\ref{Gen} needs more detailed analysis in the case of diameter two.

By Perron-Frobenius Theorem, if $\lambda_{\alpha} (G)$  is the least eigenvalue $A_{\alpha}(G)$, then
$\rho_{\alpha}(G)\ge -\lambda_{\alpha} (G)$. Thus, for a connected
irregular graph $G$ on $n$ vertices with maximum degree $\Delta$ and diameter $D$,
\[
\Delta+\lambda_{\alpha} (G)> \frac{2(1-\alpha)}{(2D-\alpha)n}.
\]
Recall that Alon and Sudakov \cite{AS} proved
that for a connected
graph $G$  on $n$ vertices  with maximum degree $\Delta$ and diameter $D$, if it  is not bipartite
(but possibly regular), then
\[
\Delta+\lambda_{0} (G)> \frac{1}{(D+1)n}.
\]

For  a connected irregular graph $G$ on $n$ vertices with maximum degree $\Delta$, minimum degree $\delta$, average $d$ and diameter $D$, Shi~\cite{Sh} showed that
\[
\rho_{0}(G)< \Delta-\frac{1}{(n-\delta)D-{D\choose 2}+\frac{1}{\Delta-d}}
\]
and
\[
\mu (G)<2\Delta-\frac{1}{(n-\delta)D-{D\choose 2}+\frac{1}{2(\Delta-d)}},
\]
where $\mu(G)$ is the the largest eigenvalue of the Laplacian matrix  of $G$. For a connected graph $G$, since $\mu(G)\le 2\rho_{1/2}(G)$, upper bounds for $2\rho_{1/2}(G)$ result in upper bounds for $\mu(G)$.

We remark that  the argument in \cite{Sh}  applies easily to prove the following result.
For completeness, however, we include a proof here.

\begin{Proposition}  \label{QH} Let $G$ be a connected irregular graph on $n$ vertices with maximum degree $\Delta$, minimum degree $\delta$, average degree $d$ and diameter $D$. For $0\le \alpha<1$, we have
\[\rho_\alpha(G)< \Delta-
\frac{1}{\frac{D(n-\delta)}{1-\alpha}-\frac{{D\choose 2}}{1-\alpha}+\frac{1}{\Delta-d}}
\]
\end{Proposition}

\begin{proof}
Let $x$ is  the Perron  vector of $A_{\alpha}(G)$. Let $x_z=\max\{x_i: i\in V(G)\}$ and  $x_w=\min\{x_i: i\in V(G)\}$. Let $v_0\dots v_p$ be a shortest path connecting $w$ and $z$, where $v_0=w$ and $v_p=z$. Now for $\ell=1,\dots, p$, by Cauchy-Schwarz inequality,
\[
\sum_{j=0}^{\ell-1}\left(x_{v_j}-x_{v_{j+1}}\right)^{2}\ge \frac{\left(x_w-x_{v_\ell}\right)^2}{\ell}.
\]
As in the proof of Theorem~\ref{Gen}, for $\ell=1,\dots, p$, we have  by Lemma~\ref{lem} that
\begin{equation} \label{ya}
\begin{split}
\Delta-\rho_{\alpha}(G)
\ge& (n\Delta-2m)x_{w}^2+(1-\alpha)\sum_{ij\in E(G)}(x_{i}-x_{j})^{2}\\
\ge &(n\Delta-2m)x_{w}^2+(1-\alpha)\frac{\left(x_w-x_{v_\ell}\right)^2}{\ell}\\
\ge & \frac{(1-\alpha)(n\Delta-2m)x_{v_\ell}^2}{\ell (n\Delta-2m)+1-\alpha}.
\end{split}
\end{equation}

It is known that $D+\Delta\le n+1$. Since $\delta<\Delta$, we have  $D+\delta\le n$.
Since  $1\le p\le D$ and $\delta\le d_w\le \Delta$, we have
\[
p(n-d_w)-{p\choose 2}\le p(n-\delta)-{p \choose 2}\le D(n-\delta)-{D \choose 2}.
\]
Let \[
\beta=\frac{1}{\frac{p(n-d_w)-{p\choose 2}}{1-\alpha}+\frac{1}{\Delta-d}}.
 \]
Then  it suffices to show that  $\Delta-\rho_{\alpha}(G)>\beta$.

If $x_w^2>\frac{\beta}{n\Delta-2m}$, then $\Delta-\rho_{\alpha}(G)\ge(n\Delta-2m)x_{w}^2>\beta$.

If $x_{v_\ell}^2>\frac{\ell (n\Delta-2m)+1-\alpha}{(1-\alpha)(n\Delta-2m)}\beta$ for some $\ell=1,\dots,p$, then from (\ref{ya}) we have
$\Delta-\rho_{\alpha}(G) >\beta$.

If $\sum_{v\in N_G(w)}x_v^2>\frac{d_w(1-\alpha)+n\Delta-2m}{(1-\alpha)(n\Delta-2m)}\beta$, then by Lemma~\ref{lem},
\begin{eqnarray*}
\Delta-\rho_{\alpha}(G)&\ge& (n\Delta-2m)x_{w}^2+(1-\alpha)\sum_{v\in N_G(w)}(x_{v}-x_{w})^{2}\\
&\ge&\sum_{v\in N_G(w)}\left(\frac{(n\Delta-2m)x_{w}^2}{d_w}+(1-\alpha)(x_{v}-x_{w})^2\right)\\
&\ge&\sum_{v\in N_G(w)}\frac{(1-\alpha)(n\Delta-2m)}{d_w(1-\alpha)+(n\Delta-2m)}x_{v}^2>\beta.
\end{eqnarray*}

Thus, we can assume that $x_w^2\le \frac{\beta}{n\Delta-2m}$, $\sum_{v\in N_G(w)}x_v^2\le\frac{d_w(1-\alpha)+n\Delta-2m}{(1-\alpha)(n\Delta-2m)}\beta$ and $x_{v_\ell}^2\le\frac{\ell (n\Delta-2m)+1-\alpha}{(1-\alpha)(n\Delta-2m)}\beta$ for $\ell=1,\dots,p$.
Then
\begin{eqnarray*}
(n-p-d_w+1)x_z^2&\ge& 1-x_w^2-\sum_{\ell=2}^{p-1}x_{v_\ell}^2-\sum_{v\in N_G(w)}x_{v}^2\\
&\ge & 1-\frac{\beta}{n\Delta-2m}-\frac{\sum_{\ell=2}^{p-1}\left(\ell (n\Delta-2m)+1-\alpha\right)}{(1-\alpha)(n\Delta-2m)}\beta\\
&&-\frac{d_w(1-\alpha)+n\Delta-2m}{(1-\alpha)(n\Delta-2m)}\beta\\
&=&1-\left(\frac{d_w+p-1}{n\Delta-2m}+\frac{\binom{p}{2}}{1-\alpha}\right)\beta.
\end{eqnarray*}
with equality only if $x_w^2= \frac{\beta}{n\Delta-2m}$.
From (\ref{ya}), we have
\[
\Delta-\rho_{\alpha}(G)
\ge \frac{(1-\alpha)(n\Delta-2m)\left(1-\left(\frac{d_w+p-1}{n\Delta-2m}+\frac{\binom{p}{2}}{1-\alpha}\right)\beta\right)}{(p(n\Delta-2m)+1-\alpha)(n-p-d_w+1)}
=\beta.
\]
Suppose that $\Delta-\rho_{\alpha}(G)=\beta$. Then 
by Lemma~\ref{lem}, we have  $x_w= \frac{1-\alpha}{p(n\Delta-2m)+1-\alpha}x_{v_p}$.  Note that we also have $x_w^2= \frac{\beta}{n\Delta-2m}$. Thus
\[
\left(\frac{p(n\Delta-2m)+1-\alpha}{1-\alpha}\right)^2\frac{\beta}{n\Delta-2m}=x_{v_p}^2\le\frac{p (n\Delta-2m)+1-\alpha}{(n\Delta-2m)(1-\alpha)}\beta,
\]
implying that  $p(n\Delta-2m)\le 0$,
a contradiction.
Therefore  $\Delta-\rho_{\alpha}(G)>\beta$.
\end{proof}

For   a $k$-connected irregular graph $G$ on $n\ge 3$ vertices with $m$ edges and maximum degree $\Delta$,
Chen and Hou \cite{CH} (see also Shiu et al.~\cite{Shiu}) showed that
\[
\rho_{0}(G)<\Delta-\frac{(n\Delta-2m)k^2}{(n\Delta-2m)(n^2-(\Delta-k+2)(n-k))+nk^2},
\]
and Shiu et al. \cite{Shiu} showed  that
\[
2\rho_{1/2}(G)<2\Delta-\frac{(n\Delta-2m)k^2}{2(n\Delta-2m)(n^2-(\Delta-k+2)(n-k))+nk^2}.
\]

The argument in  \cite{CH,Shiu} leads easily to the following result. For completeness, however, we include a proof here.

\begin{Proposition}  \label{Shao} Let $G$ be a $k$-connected irregular graph on $n$ vertices with $m$ edges, maximum degree $\Delta$. For $0\le \alpha<1$, we have
\[
\rho_\alpha(G)<\Delta-\frac{(1-\alpha)(n\Delta-2m)k^2}{(n\Delta-2m)(n^2-(\Delta-k+2)(n-k))+(1-\alpha)nk^2}.
\]
\end{Proposition}

\begin{proof} Let \[
\beta=\frac{(1-\alpha)(n\Delta-2m)k^2}{(n\Delta-2m)(n^2-(\Delta-k+2)(n-k))+(1-\alpha)nk^2}.\]
Note that
$n^2-(\Delta-k+2)(n-k)\ge nk>k^2$. Then $\beta<1$.

Let $x$ is  the Perron  vector of $A_{\alpha}(G)$. Let $x_z=\max\{x_i: i\in V(G)\}$. Then $x_z >\frac{1}{\sqrt{n}}$.

If $d_z<\Delta$, then as in the proof of Theorem~\ref{Gen}, we have
$\rho_{\alpha}(G) \le \Delta-1<\Delta-\beta$.

Assume that $d_z=\Delta$. Let $w$ be a vertex of $G$ such that $x_w=\min\{x_i: i\in V(G)\}$. From the eigenequation at $w$, we have
$d_w<\Delta$.  As in \cite{CH,Shiu},
%
\[
\sum_{s=1}^{k}\sum_{ij\in E(Q_s)}(x_i-x_j)^2
\ge \frac{k^2}{n-\Delta+2k-2}(x_w-x_z)^2.
\]
Thus, by the argument in Theorem~\ref{Gen}, we have
\begin{eqnarray*}
\Delta-\rho_{\alpha}(G)
& \ge & (n\Delta-2m)x_{w}^2+\frac{(1-\alpha)k^2}{n-\Delta+2k-2}(x_w-x_z)^2.
\end{eqnarray*}
Note that $x_w\ne x_z$ as $G$ is irregular.

If $x_w^2\ge\frac{\beta}{n\Delta-2m}$, then  $\Delta-\rho_{\alpha}(G)
> (n\Delta-2m)x_w^2
\ge \beta$,
as desired.

Assume that $x_w^2<\frac{\beta}{n\Delta-2m}$. By Lemma~\ref{lem},
\begin{equation} \label{wa}
\Delta-\rho_{\alpha}(G)\ge \frac{(1-\alpha)(n\Delta-2m)k^2}{(n\Delta-2m)(n-\Delta+2k-2)+(1-\alpha)k^2}x_z^2.
\end{equation}

If  $k=1$, then
$(n-1)x_z^2\ge 1-x_w^2>1-\frac{\beta}{n\Delta-2m}$, and thus from (\ref{wa}), we have
\begin{eqnarray*}
\Delta-\rho_{\alpha}(G)
&>& \frac{(1-\alpha)(n\Delta-2m)}{(n\Delta-2m)(n-\Delta)+1-\alpha}\cdot\frac{1}{n-1}\cdot\left(1-\frac{\beta}{n\Delta-2m}\right)\\
&=&\beta \cdot\frac{(n\Delta-2m)\left(\frac{n^2}{n-1}-(\Delta+1)\right)+1-\alpha}{(n\Delta-2m)(n-\Delta)+1-\alpha}\\
&>&\beta,
\end{eqnarray*}
as desired.

Now assume that $k\ge2$.
Since $d_w\ge k$, we may choose $k-1$ vertices, say $v_1,\dots, v_{k-1}$,  in $N_G(w)$ different from $z$. 
Then as above, we have
\begin{eqnarray*}
\Delta-\rho_{\alpha}(G)
&\ge&(n\Delta-2m)x_{w}^2+(1-\alpha)\sum_{i=1}^{k-1}(x_{v_i}-x_{w})^{2}\\
&=&\sum_{i=1}^{k-1}\left(\frac{n\Delta-2m}{k-1}x_{w}^2+(1-\alpha)(x_{v_i}-x_{w})^{2}\right)\\
&\ge&\sum_{i=1}^{k-1}\frac{(1-\alpha)(n\Delta-2m)}{n\Delta-2m+(1-\alpha)(k-1)}x_{v_i}^2\\
&=&\frac{(1-\alpha)(n\Delta-2m)}{n\Delta-2m+(1-\alpha)(k-1)}\sum_{i=1}^{k-1}x_{v_i}^2.
\end{eqnarray*}

If $\sum_{i=1}^{k-1}x_{v_i}^2>\frac{n\Delta-2m+(1-\alpha)(k-1)}{(1-\alpha)(n\Delta-2m)}\beta$, then
$\Delta-\rho_{\alpha}(G)>\beta$, 
as desired.

Assume that $\sum_{i=1}^{k-1}x_{v_i}^2\le \frac{n\Delta-2m+(1-\alpha)(k-1)}{(1-\alpha)(n\Delta-2m)}\beta$. Recall that $x_w^2<\frac{\beta}{n\Delta-2m}$.
Then
\[
(n-k)x_z^2\ge 1-x_w^2-\sum_{i=1}^{k-1}x_{v_i}^2> 1-\frac{n\Delta-2m+(1-\alpha)k}{(1-\alpha)(n\Delta-2m)}\beta.
\]
Therefore, from (\ref{wa}), we have
\begin{eqnarray*}
\Delta-\rho_{\alpha}(G)&\ge& \frac{(1-\alpha)(n\Delta-2m)k^2}{(n\Delta-2m)(n-\Delta+2k-2)+(1-\alpha)k^2}x_z^2\\
&>&\frac{(1-\alpha)(n\Delta-2m)k^2}{(n\Delta-2m)(n-\Delta+2k-2)+(1-\alpha)k^2} \\
&&\cdot\frac{1}{n-k}\left(1-\frac{n\Delta-2m+(1-\alpha)k}{(1-\alpha)(n\Delta-2m)}\beta\right)\\
&=&\beta,
\end{eqnarray*}
as desired.
\end{proof}

By direct check, the upper bound in Theorem~\ref{Gen} is less than or equal to the upper  bound in Proposition~\ref{QH} if and only if
\[
(\Delta-d)\left(2D\delta+D(D-1)-\alpha n\right)\le 2(1-\alpha),
\]
the upper bound in Theorem~\ref{Gen} is less than or equal to the upper bound in Proposition~\ref{Shao} if and only if
\[
2n^2+\frac{2(1-\alpha)nk^2}{n\Delta-2m}\ge n(2D-\alpha)k^2+2(\Delta-k+2)(n-k),
\]
and the upper bound in Proposition~\ref{QH} is less than or equal to the upper bound in Proposition~\ref{Shao} if and only if
\[
k^2D(2n-2\delta-D+1)\le 2n^2-2(\Delta-k+2)(n-k).
\]

For a graph $G$ with $u\in V(G)$, let $R_u=V(G)\setminus N_G(u)$. The following result concerning the domination number unifies the results in  \cite{SAH,XZ} on spectral radius and signless Laplacian spectral radius of a graph. We note that the bound is independent of the parameter $\alpha$.

\begin{Theorem} Let $G$ be a graph with $n$ vertices and domination number $\gamma$, where $1\le\gamma \le n-1$. For $0\le \alpha<1$, we have $\rho_{\alpha}(G)\le n-\gamma$ with equality if and only if
$G\cong K_{n-\gamma+1}\cup(\gamma-1)K_1$ or when $\gamma\ge 2$ and $n-\gamma$ is even, $G\cong\overline{\frac{n-\gamma+2}{2}K_2}\cup(\gamma-2)K_1$.
\end{Theorem}

\begin{proof}  Let $\Delta$ be the maximum degree of $G$. For $u\in V(G)$ with $d_G(u)=\Delta$, it is easily seen that $R_u$ is a dominating set of $G$, and thus  $\gamma\le |R_u|=n-\Delta$, implying that $\Delta(G)\le n-\gamma$ with equality if and only if $R_u$ is a minimum  dominating set of $G$. Thus
$\rho_\alpha(G)\le \Delta\le n-\gamma$.

Suppose that  $\rho_\alpha(G)=n-\gamma$. Then $\Delta=n-\gamma$.  Obviously, $\rho_\alpha(G)=\rho_\alpha(G_1)$ for some nontrivial component $G_1$ of $G$. Let $\Delta_1$ be the maximum degree of $G_1$.
Note that $\rho_\alpha(G)=\rho_\alpha(G_1)\le \Delta_1\le \Delta=n-\gamma$. Thus  $G_1$ is regular, $\Delta_1=n-\gamma$, and
 $R_u$ is a minimum dominating set of $G$ for some $u\in V(G_1)$. Thus, $R_u$ is an independent set of $G$, and if $G$ is not connected, then any component different from $G_1$ is trivial.
 If $\gamma=1$, then  $G\cong K_n$. Suppose  that $\gamma\ge 2$.

Suppose that $d_{G_1}(u)\le |V(G_1)|-3$ for some $u\in V(G_1)$. Then there exists $v, w\in V(G_1)$ such that $uv, uw\not\in E(G_1)$. Since $G_1$ is $(n-\gamma)$-regular and
$R_u$ is an independent set of $G$, $v$ and $w$ are both adjacent to each vertex of $N_{G_1}(u)$, implying that,  for a vertex $z\in N_{G_1}(u)$, $(R_u\setminus \{v,w\})\cup\{z\}$ is a dominating set of $G$ with cardinality $\gamma-1$, a contradiction. Thus $d_{G_1}(u)=|V(G_1)|-1$ or $|V(G_1)|-2$.
If $d_{G_1}(u)=|V(G_1)|-1$, then since $G_1$ is $(n-\gamma)$-regular, we have  $G_1\cong K_{n-\gamma+1}$, and thus $G\cong K_{n-\gamma+1}\cup(\gamma-1)K_1$.
Suppose that $d_{G_1}(u)=|V(G_1)|-2$. Then there is unique vertex, say $v$,  in $V(G_1)\setminus\{u\}$ that is not adjacent to $u$, and $N_{G_1}(u)=N_{G_1}(v)$.  For any  $w\in N_{G_1}(u)$, since $w$ is adjacent to both $u$ and $v$, there is a unique vertex in  $N_{G_1}(u)\setminus\{w\}$ that is not adjacent to $w$ in $G_1$. Thus  $n-\gamma$ is even,  $G_1\cong \overline{\frac{n-\gamma+2}{2}K_2}$, and thus $G\cong\overline{\frac{n-\gamma+2}{2}K_2}\cup(\gamma-2)K_1$.

If $G\cong K_{n-\gamma+1}\cup(\gamma-1)K_1$, or if $\gamma\ge 2$, $n-\gamma$ is even and $G\cong\overline{\frac{n-\gamma+2}{2}K_2}\cup(\gamma-2)K_1$, then $G$ has a unique nontrivial regular component of degree $n-\gamma$, and thus $\rho_\alpha(G)=n-\gamma$.
\end{proof}

If $T$ is a tree on $n$ vertices, then, for $0\le \alpha<1$, we have by Corollary~\ref{ss} that $\rho_\alpha(T)\leq \rho_\alpha(S_n)$
with equality if and only if $T\cong S_n$, see \cite{NP}.

For $n\geq 4$ and $1\leq a\leq \left \lfloor \frac{n-2}{2}\right \rfloor$, let $D_{n,a}$ be the tree obtained from vertet-disjoint $S_{a+1}$ with center $u$
and $S_{n-a-1}$ with center $v$ by adding an edge $uv$.

\begin{Theorem} Let $T$ be a tree on $n\ge 4$ vertices. Suppose that $T\ncong S_n$. Then for $0\le \alpha<1$, $\rho_\alpha(T)\leq \rho_\alpha(D_{n,1})$
with equality if and only if $T\cong D_{n,1}$.
\end{Theorem}

\begin{proof} It is trivial if $n=4$. Suppose that $n\ge5$. Let $T$ be a tree  with maximum $\alpha$-spectral radius among trees on $n$ vertices except the star $S_n$.

Let $d$ be the diameter of $T$. Since $T\ncong S_n$, we have $d\ge3$. Suppose that  $d\ge4$. Let $v_0 v_{1}\dots v_{d}$ be a diametral path of $T$.
Let $N_1=N_T(v_{d-1})\setminus \{v_{d-2}\}$
Let $T'=T-\{v_{d-1}v: v\in N_1\}+\{v_{d-2}v: v\in N_1\}$.
Obviously, $T'\ncong S_n$. By Corollary~\ref{ss}, we have $\rho_\alpha(T)<\rho_\alpha(T')$, a contradiction.
Thus $d=3$ and $T\cong D_{n,a}$, where $1\leq a\leq \left \lfloor \frac{n-2}{2}\right \rfloor$. By Lemma~\ref{moving edge}, we have  $a=1$ and  $T\cong D_{n,1}$.
%
\end{proof}

For positive integer $p$ and a graph $G$ with $u\in V(G)$, let $G(u;p)$  be the graph obtained from $G$ by attaching a pendant path of
length $p$ at $u$, and let   $G(u,0)=G$.

For nonegative integers $p$, $q$ and a graph $G$, let $G_u(p,q)$ or simply $G_{p,q}$ be the graph $H(u;q)$ with $H=G(u;p)$.
%
%
%
Nikiforov and Rojo~\cite{NO} conjectured  that $\rho_\alpha(G_{p, q})>\rho_\alpha(G_{p+1, q-1})$ for a nontrivial connected graph $G$ and integers $p$ and $q$ with $p\geq q\geq 2$, and mentioned  that they can show it is true when $\rho_\alpha(G_{p+1, q-1})\ge \frac{9}{4}$.  We show that it is really true.

\begin{Theorem}\label{pq}
Let $G$ be a connected graph with $|E(G)|\geq1$ and $u\in V(G)$. For integers $p\geq q\geq 1$ and $0\le \alpha<1$, $\rho_{\alpha}(G_{u}(p, q))> \rho_{\alpha}(G_{u}(p+1, q-1))$.
\end{Theorem}

\begin{proof} Let $u  u_1\dots u_{p+1}$ and $uv_1\dots v_{q-1}$ be the two pendant paths in  $G_u(p+1, q-1)$ at $u$ of lengths $p+1$ and $q-1$, respectively. Let $x$ be the Perron  vector of $A_{\alpha}(G_u(p+1, q-1))$. Let $v_0=u$.
Suppose that $\rho_{\alpha}(G_{u}(p, q))\le \rho_{\alpha}(G_{u}(p+1, q-1))$.

\noindent {\bf Claim.} For all $i=0,1, \dots, q-1$, $x_{u_{p-i}}>x_{v_{q-i-1}}$.

If $x_{v_{q-1}}\ge x_{u_p}$, then for  $H=G_u(p+1, q-1)-u_pu_{p+1}+v_{q-1}u_{p+1}$, we have $H \cong G_{u}(p, q)$, and thus  by Lemma~\ref{moving edge},  $\rho_{\alpha}(G_{u}(p, q))=\rho_{\alpha}(H)> \rho_{\alpha}(G_{u}(p+1, q-1))$, a contradiction. Thus $x_{u_p}>x_{v_{q-1}}$. This proves the claim for $i=0$. If $q=1$, then $i=0$ and the claim follows.
Suppose that $q\ge 2$, and  $x_{u_{p-i}}>x_{v_{q-i-1}}$ where  $0\le i\le q-2$. If  $x_{v_{q-(i+1)-1}}\ge x_{u_{p-(i+1)}}$, then for
\begin{eqnarray*}
H'&=&G_{u}(p+1, q-1)-\{u_{p-(i+1)}u_{p-i}, v_{q-(i+1)-1}v_{q-i-1}\}\\
   &&+\{u_{p-i}v_{q-(i+1)-1}, u_{p-(i+1)}v_{q-i-1}\},
\end{eqnarray*}
 we have $H'\cong G_{u}(p, q)$ and thus by Lemma~\ref{switching} that $\rho_{\alpha}(G_{u}(p, q))=\rho_{\alpha}(H')> \rho_{\alpha}(G_{u}(p+1, q-1))$, a contradiction.  Thus  $x_{u_{p-(i+1)}}> x_{v_{q-(i+1)-1}}$.
Therefore, the claim follows.

%
%

By the claim for $i=q-1$, we have $x_{u_{p-(q-1)}}>x_{u}$. Since $G_{u}(p+1, q-1)-\{uw: uw\in E(G)\}+ \{u_{p-(q-1)}w: uw\in E(G)\}\cong G_u(p, q)$, we have by Lemma~\ref{moving edge} that $\rho_{\alpha}(G_u(p, q))> \rho_{\alpha}(G_u(p+1, q-1))$, a contradiction.

Therefore $\rho_{\alpha}(G_{u}(p, q))> \rho_{\alpha}(G_{u}(p+1, q-1))$.
\end{proof}

Let $G$ be a connected graph with $uv\in E(G)$. For nonnegative integers $p$ and $q$, let $G_{u,v}(p, q)$ be the graph $H(v;q)$ with $H=G(u;p)$.
It was conjectured in~\cite{NO} that if the degrees of $u$ and $v$ are at least two in $G$, then
for $p\ge q \ge 2$ and $0\le \alpha <1$, $\rho_{\alpha}(G_{u,v}(p,q))>\rho_{\alpha}(G_{u,v}(p+1,q-1))$. Now we show that this is also indeed true.

\begin{Theorem} \label{pq+} Let $G$ be a connected graph, and let $u$ and $v$ be adjacent vertices of $G$ of degree at least $2$. For $p\ge q \ge 1$ and $0\le \alpha <1$, $\rho_{\alpha}(G_{u,v}(p,q))>\rho_{\alpha}(G_{u,v}(p+1,q-1))$.
\end{Theorem}

\begin{proof} Let $u u_1 \dots u_{p+1}$ and $v v_1 \dots v_{q-1}$ be the two pendant paths at $u$ and $v$ in $G_{u,v}(p+1,q-1)$, respectively. Let $x$ be the Perron vector of $A_{\alpha}(G_{u,v}(p+1,q-1))$. Let $u_0=u, v_0=v$. Suppose that $\rho_{\alpha}(G_{u,v}(p,q))\le\rho_{\alpha}(G_{u,v}(p+1,q-1))$.

By argument as in the proof of Theorem~\ref{pq}, we have $x_{u_{p-i}}>x_{v_{q-i-1}}$  for all $i=0,1,\dots, q-1$. Thus  $x_{u_{p-(q-1)}}>x_{v}$. Let $G'=G_{u,v}(p+1,q-1)-\{vw: vw\in E(G)\}+\{u_{p-(q-1)}w: vw\in E(G)\}$. By Lemma~\ref{moving edge}, we have $\rho_{\alpha}(G')> \rho_{\alpha}(G_{u}(p+1, q-1))$. If $p=q$, then $G'\cong G_{u,v}(p,q)$ and thus $\rho_{\alpha}(G_{u,v}(p,q))>\rho_{\alpha}(G_{u,v}(p+1,q-1))$, a contradiction. Thus $p>q$.
Let $x'$ be the Perron vector of $A_{\alpha}(G')$.
Note that $x'_{u_{p-(q-1)}}> x'_{v}$; Otherwise, we have by  Lemma~\ref{moving edge} that $\rho_{\alpha}(G')<\rho_{\alpha}(G_{u}(p+1, q-1))$, a contradiction.

If $x'_{u_{p-q}}\ge x'_{u}$, then since $G'-\{uw: uw\in E(G)\}+\{u_{p-q}w: uw\in E(G) \}\cong G_{u,v}(p,q)$, we have by Lemma~\ref{moving edge} that $\rho_{\alpha}(G_{u,v}(p,q))>\rho_{\alpha}(G')>\rho_{\alpha}(G_{u,v}(p+1,q-1))$, a contradiction. Thus we may assume that $x'_{u_{p-q}}< x'_{u}$. Since $G'-\{uv, u_{p-q}u_{p-(q-1)}\}+\{uu_{p-(q-1)}, vu_{p-q}\}\cong G_{u,v}(p,q)$, we have by Lemma~\ref{switching} that $\rho_{\alpha}(G_{u,v}(p,q))>\rho_{\alpha}(G')>\rho_{\alpha}(G_{u,v}(p+1,q-1))$, also a contradiction.

Therefore, $\rho_{\alpha}(G_{u,v}(p,q))>\rho_{\alpha}(G_{u,v}(p+1,q-1))$.
\end{proof}

%

For $3\le d\le n-1$, let $T_{n,d}$ be the tree obtained from a path $v_0v_1 \cdots v_d$ with length $d$ by attaching $n-1-d$ pendant edges at vertex $v_{\lfloor \frac{d}{2}\rfloor}$.

\begin{Theorem}  Let $T$ be a tree with $n$ vertices and diameter $d\ge 3$. For $0\le \alpha<1$, $\rho_{\alpha}(T)\le\rho_{\alpha}(T_{n, d})$  with equality if and only if $T\cong T_{n, d}$.
\end{Theorem}

\begin{proof} Let $T$ be a tree with maximum $\alpha$-spectral radius among trees with $n$  vertices and diameter $d$. Let $P=v_0\dots v_{d}$ be a diametral  path of $T$. For any $u\in V(T)$, let $d_T(u,P)=\min\{d_T(u, v_i): i=0, \dots, d\}$.

Suppose that $uv$ is an edge outside $P$ that is not a pendant edge. Assume that $d_T(u, P)<d_T(v, P)$. Let $w$ be the vertex on $P$ with $d_T(u, P)=d_T(u, w)$.
Let $T^*=T-\{vz: vz\in E(T), z\ne u\}+ \{wz: vz\in E(T), z\ne u\}$ if  $x_w\ge x_v$,   
and $T^*=T-\{wz: wz\in E(T)\setminus \{e\}\}+ \{vz: zw\in E(T)\setminus \{e\}\}$ otherwise,  where $e$ is the edge incident with  $w$ in the path connecting $w$ and $v$. Obviously, $T^*$ is
a tree with $n$ vertices and diameter $d$.
By Lemma~\ref{moving edge}, $\rho_{\alpha}(T^*)> \rho_{\alpha}(T)$, a contradiction. Every edge outside $P$ is a pendant edge at some vertex of $P$ except $v_0$ and $v_d$.

Suppose that there are two vertices, say $u$ and $v$, on $P$ with degree greater than two.  We may assume that $x_u\ge x_v$. Let $T^*=T-\{vz: vz\in E(T)\setminus E(P)\}+ \{uz: vz\in E(T)\setminus E(P)\}$.  
By Lemma~\ref{moving edge}, we have $\rho_{\alpha}(T^*)> \rho_{\alpha}(T)$, a contradiction. It follows that there is at most one vertex on $P$ with degree greater than two.

Therefore $T$ is obtainable from $P$ by attaching $n-d-1$ pendant edges at a vertex different from $v_0$ and $v_d$. By  Theorem~\ref{pq}, we have $T\cong T_{n, d}$.
\end{proof}

It is known that $T_{n,d}$ is the unique tree with maximum $0$-spectral radius among trees with $n$ vertices and diameter $d\ge 3$, see \cite{GS,SZ}.

\section{Difference between maximum degree and $\alpha$-spectral radius}

Recall that for a graph $G$ with maximum degree $\Delta$ and $0\le \alpha<1$,  $\rho_{\alpha}(G)\le \Delta$ with equality if and only if  $G$ has a component that is regular of degree $\Delta$. Let $\gamma_{\alpha}(G)=\Delta-\rho_{\alpha}(G)$. We may view $\gamma_{\alpha}(G)$ as a measure of irregularity of the graph $G$. The case when $\alpha=0$ has been studied in \cite{Ob}.

\begin{Theorem} \label{irr1} Let $G$ be a graph on $n\ge 2$ vertices. For $0\le \alpha<1$, we have
\[
\gamma_{\alpha}(G)\le n-1-\frac{\alpha n}{2}-\frac{\sqrt{\alpha^2n^2+4(1-2\alpha)(n-1)}}{2}
\]
with equality if and only if $G\cong S_n$.
\end{Theorem}

\begin{proof} Let $\Delta$ be the maximum degree of $G$. Then $S_{\Delta+1}$ is a  subgraph of $G$. By  \cite[Crollary 2.2, p.~38]{Mi},
\[
\rho_{\alpha}(G)\ge \rho_{\alpha}(S_{\Delta+1}) 
\]
with equality when $G$ is connected (i.e., $A_{\alpha}(G)$ is irreducible) if and only if $G\cong S_{\Delta+1}$. From \cite{Ni1}, we have
\[
 \rho_{\alpha}(S_{\Delta+1})=\frac{\alpha(\Delta+1)+\sqrt{\alpha^2(\Delta+1)^2+4(1-2\alpha)\Delta}}{2}.
\]
Thus
\[
\gamma_{\alpha}(G)\le f(\Delta)
\]
with
\[
f(t)=t-\frac{\alpha(t+1)}{2}-\frac{\sqrt{\alpha^2(t+1)^2+4(1-2\alpha)t}}{2}.
\]
Note that
\[
f'(t)=1-\frac{\alpha}{2}-\frac{\alpha^2(t+1)+2(1-2\alpha)}{2\sqrt{\alpha^2(t+1)^2+4(1-2\alpha)t}}.
\]
It may be seen that $f'(t)> 0$ if and only if $g(t)>0$, where
\[
g(t)=\alpha^2t^2+2(\alpha^2-4\alpha+2)t-\alpha^2+3\alpha-1.
\]
Since $0\le \alpha<1$, we have
$-\frac{2(\alpha^2-4\alpha+2)}{2\alpha^2}<1$.
Thus $g(t)$ is strictly increasing for $t\ge 1$. Now it follows that for $t\ge1$,
\[
g(t)\ge g(1)=(\alpha-1)(2\alpha-3)
>0,
\]
or equivalently, $f'(t)>0$.
Thus $f(t)$ is strictly increasing for $t\ge 1$. Therefore
\[
\gamma_{\alpha}(G)\le f(\Delta)\le f(n-1)=n-1-\frac{\alpha n}{2}-\frac{\sqrt{\alpha^2n^2+4(1-2\alpha)(n-1)}}{2}
\]
with equalities if and only if $\Delta=n-1$ (implying that $G$ is connected) and $\rho_{\alpha}(G)=\rho_{\alpha}(S_{\Delta+1})$, or equivalently, $G\cong S_n$.
\end{proof}

The case $\alpha=0$ in previous theorem has been obtained in \cite{Ob}.

For connected graph $G$, if there is an automorphism $\sigma$ such that $\sigma(u)=v$, then $x_u=x_v$, where $x$ is the Perron vector of $A_{\alpha}(G)$ and $0\le \alpha<1$, see \cite[Proposition 16]{Ni1}.

For $n\ge 3$, let $S_n+e$ be the unicyclic graph obtained from the star by adding an edge to connect two vertices of degree one.

Let $x$ be the Perron vector of $A_{\alpha}(S_n+e)$.  Let $x_1$ be the entry of $x$ corresponding to the vertex of degree $n-1$. The entry of $x$ corresponding to the either vertex of degree $2$ is equal, which is denoted by $x_2$, the entry of each vertex of degree $1$ is equal, which is denoted by $x_3$. Let $\rho=\rho_{\alpha}(S_n+e)$.  Thus
\begin{eqnarray*}
(\rho-\alpha(n-1))x_1&=&2(1-\alpha)x_2+(n-3)(1-\alpha)x_3,\\
(\rho-1-\alpha)x_2&=&(1-\alpha)x_1,\\
(\rho-\alpha)x_3&=&(1-\alpha)x_1.
\end{eqnarray*}
Therefore $h(\rho)=0$ with
\begin{eqnarray*}
h(t)&=&t^3-(\alpha (n+1)+1)t^2+((\alpha^2+3\alpha-1)(n-1)+\alpha(\alpha+1))t\\
&&+(1-2\alpha)(\alpha+1)(n-1)-2(1-\alpha)^2.
\end{eqnarray*}
It follows that $\rho_\alpha((S_n+e))$ is the largest root of $h(t)=0$.

\begin{Theorem} \label{irr2}
Let $G$ be a unicyclic graph with $n\ge 4$ vertices. For $0\le \alpha<1$, we have
\[
\gamma_{\alpha}(G)\le \gamma_{\alpha}(S_n+e)
\]
with equality if and only if $G\cong S_n+e$.
\end{Theorem}

\begin{proof} Let $\Delta$ be the maximum degree of $G$. Let
\[
t_0=1+\frac{\alpha(n-1)}{2}+\frac{\sqrt{\alpha^2(n-1)^2+4(1-2\alpha)(n-2)}}{2}.
\]

If $\Delta\le n-2$, then by the argument as in the proof of Theorem~\ref{irr1}, we have
\[
\gamma_{\alpha}(G)\le f(\Delta)\le f(n-2)=n-1-t_0.
\]
If $\Delta=n-1$, then $G\cong S_n+e$, and
\[
\gamma_{\alpha}(G)=n-1-\rho_{\alpha}(S_n+e),
\]
where $\rho_{\alpha}(S_n+e)$ is the largest root of $h(t)=0$.

Let $t_1$ be the larger root of  $h'(t)=0$, i.e.,
\[
t_1=\frac{\alpha(n+1)+1+\sqrt{\alpha^2n^2-(\alpha^2+7\alpha-3)n+\alpha^2+8\alpha-2}}{3}.
 \]
It may be checked that $t_0>t_1$. Thus $h(t)$ is strictly increasing for $t\ge t_0$.

\noindent
{\bf Case 1.} $0\le \alpha\le \frac{1}{2}$.

Note that
\begin{eqnarray*}
h(t_0)&=&\alpha(1-\alpha)(n-2)\frac{\alpha(n-1)+\sqrt{\alpha^2(n-1)^2+4(1-2\alpha)(n-2)}}{2}\\
&&+(3n-8)\alpha^2+(14-5n)\alpha+2n-6.
\end{eqnarray*}
We  view $h(t_0)$ as a function of $n$, denoted by $H(n)$. Then
\begin{eqnarray*}
2H'(n)&=&(2n-3)\alpha^2(1-\alpha)+\alpha(1-\alpha)\sqrt{\alpha^2(n-1)^2+4(1-2\alpha)(n-2)}\\
&&+\frac{\alpha(1-\alpha)(n-2)(\alpha^2(n-1)+2(1-2\alpha))}{\sqrt{\alpha^2(n-1)^2+4(1-2\alpha)(n-2)}}+6\alpha^2-10\alpha+4.
\end{eqnarray*}
For $n\ge 4$, since $0\le \alpha \le \frac{1}{2}$, we have $H'(n)>0$, and
thus $H(n)$ is strictly increasing for $n\ge 4$. Therefore
\[
h(t_0)=H(n)\ge H(4)=\alpha(1-\alpha)\sqrt{9\alpha^2+8(1-2\alpha)}+(1-\alpha)(3\alpha^2-4\alpha+2)
>0.
\]

\noindent {\bf Case 2.} $\frac{1}{2}< \alpha <1$.

Note that
\[
\sqrt{\alpha^2(n-1)^2+4(1-2\alpha)(n-2)}\ge \alpha(n-3)+\frac{2(1-\alpha)^2}{\alpha}.
\]
Then
\[
h(t_0)\ge \alpha^2(1-\alpha)(n-2)^2+((1-\alpha)^3+3\alpha^2-5\alpha+2)(n-2)-2(1-\alpha)^2.
\]
Let $H(n)$ be the expression in the right hand side of the above inequality, which is a function of $n$.
Then
\[
H'(n)=2\alpha^2(1-\alpha)n+3\alpha^3+2\alpha^2-8\alpha+3,
\]
which is strictly increasing for $n\ge 4$. It follows that 
\[
H'(n)\ge H'(4)=(1-\alpha)(5\alpha^2-5\alpha+3)
>0.
\]
Therefore
\[
h(t_0)\ge H(n)\ge H(4)=2(1-\alpha)(3\alpha^2-4\alpha+2)>0.
\]

Now combining Cases 1 and 2, we have $h(t_0)>0$. Thus, for $t\ge t_0$, we have $h(t)>h(t_0)>0$, implying that $\rho_{\alpha}(S_n+e)<t_0$. Now the result follows.
\end{proof}

\begin{Theorem} Let $G$ be a graph on $n\ge 4$ vertices. If $G$ is not bipartite, then for $0\le \alpha<1$, we have
\[
\gamma_{\alpha}(G)\le \gamma_{\alpha}(S_n+e)
\]
with equality if and only if $G\cong S_n+e$.
\end{Theorem}

\begin{proof} Let $\Delta$ be the maximum degree of $G$.  If $\Delta=n-1$, then $S_n+e$ is a subgraph of $G$, and thus
\[
\gamma_{\alpha}(G)=n-1-\rho_{\alpha}(G)\le n-1-\rho_{\alpha}(S_n+e)
\]
with equality if and only if $G\cong S_n+e$. Now the result follows as in the proof of Theorem~\ref{irr2}.
\end{proof}

\section {Comments}

Because of the work of  Nikiforov~\cite{Ni1},  we may study the (adjacency) spectral properties and signless Laplacian spectral properties of a graph in a unified way.
Thus, we may also study those parameters based on the spectrum of $A_{\alpha}(G)$ of a graph $G$. We give two such examples. 

Let  $0\le \alpha<1$.
Let $G$ be a graph with $n$ vertices and $m$ edges.
Let $\lambda_1, \dots, \lambda_n$ be the eigenvalues of $A_{\alpha}(G)$, arranged in a non-increasing manner. Obviously, $\lambda_1=\rho_{\alpha}(G)$.

The first is the $\alpha$-energy of $G$, which is defined as
\[
\mathcal{E}_{\alpha}(G)=\sum_{i=1}^n \left|\lambda_i-\frac{2\alpha m}{n}\right|.
\]
Note that $\mathcal{E}_0(G)$ is the energy of $G$, which has been studied extensively \cite{Gu,LSG}, and $\mathcal{E}_{1/2}(G)$ is half of the signless Laplacian energy of $G$, which has received some attention in recent years \cite{ACG}. Let $A_{\alpha}=A_{\alpha}(G)$.
Note that $\sum_{i=1}^n\lambda_i=\mbox{tr}(A_{\alpha})=2\alpha m$  and $\sum_{i=1}^n\lambda_i^2=\mbox{tr}( A_\alpha^2)=2(1-\alpha)^2m+\alpha^2Z(G)$, where $Z(G)=\sum_{u\in V(G)}d_G(u)^2$.
By Cauchy-Schwarz inequality, we have
\begin{eqnarray*}
\mathcal{E}_{\alpha}(G)
&\le&\sqrt{n\sum_{i=1}^n\left|\lambda_i-\frac{2\alpha m}{n}\right|^2}\\
&=&\sqrt{n\sum_{i=1}^n\lambda^2_i-4\alpha^2m^2}\\
&=&\sqrt{2(1-\alpha)^2mn+\alpha^2(nZ(G)-4m^2)}\,.
\end{eqnarray*}
For $\mathcal{E}_0(G)$, this is just McClelland's upper bound in \cite{Mc}. 
 On the other hand, it is easily seen that $\mathcal{E}_{\alpha}(G)\ge 2(\lambda_1-\frac{2\alpha m}{n})$.
Note that $\lambda_1\ge \rho_0(G)$ \cite[Proposition~18]{Ni1}. Lower bounds for $\rho_0(G)$, for example,  $\lambda_1\ge \sqrt{\frac{Z(G)}{n}}\ge \frac{2m}{n}$, may be used to derive   lower bounds for
$\mathcal{E}_{\alpha}(G)$. Furthermore, let $s_i=\lambda_i-\frac{2\alpha m}{n}$ for $i=1, \dots, n$. Then $\sum_{i=1}^n s_i=0$,
implying that $\sum_{i=1}^ns_i^2\le 2\sum_{1\le i<j\le n}|s_i| |s_j|$. Thus $\mathcal{E}_{\alpha}(G)^2=\sum_{i=1}^ns_i^2+2\sum_{1\le i<j\le n}|s_i| |s_j|\ge 2\sum_{i=1}^ns_i^2$, i.e.,
\[
\mathcal{E}_{\alpha}(G)\ge \sqrt {2\left(2(1-\alpha)^2m+\alpha^2\left(Z(G)-\frac{4m^2}{n}\right)   \right)}\,.
\]


The second one is the $\alpha$-Estrada index of a graph $G$, defined as $EE_{\alpha}(G)=\sum_{i=1}^ne^{\lambda_i}$. Obviously, $EE_0(G)$ is just the much studied Estrada index of $G$, see, e.g., \cite{de,Es}. %
Note also that $EE_{1/2}(G)$ is somewhat different from the so called signless Laplacian Estrada index \cite{AB, GM}, which is defined to be $\sum_{i=1}^ne^{2\lambda_i}$ (with $\lambda_i$'s being the eigenvalues of $A_{1/2}(G)$).
For a graph $G$ with $n$ vertices and $m$ edges, it is easily seen that
\[
EE_{\alpha}(G)=n+2\alpha m+\sum_{k\ge 2} \frac{\sum_{i=1}^n\lambda_i^k}{k!}.
\]
As in \cite{ZG}, we have $\sum_{i=1}^n\lambda_i^k\le \left(\sum_{i=1}^n\lambda_i^2\right)^{k/2}=\left(\sqrt{2(1-\alpha)^2m+\alpha^2Z(G)}\right)^k$ for $k\ge 2$. Thus
\[
EE_{\alpha}(G)\le n-1+2\alpha m-\sqrt{2(1-\alpha)^2m+\alpha^2Z(G)}+e^{\sqrt{2(1-\alpha)^2m+\alpha^2Z(G)}}.
\]
\vspace{3mm}

\bigskip

\noindent {\bf Acknowledgement.} 
This work was supported by
the National Natural Science Foundation of China (No.~11671156).

%
%
%

\end{document}